\documentclass[12pt, reqno]{amsartvera}
\usepackage{amsmath, amsthm, amssymb}
\usepackage{amsfonts}
\usepackage{mathrsfs}
\usepackage[all]{xy}
\usepackage{wasysym}
\usepackage{bbm}
\usepackage{textcomp}
\usepackage{verbatim}
\usepackage{charter}
\usepackage{fancyhdr}
\usepackage{epsfig}
\usepackage{pb-diagram,pb-xy}      
\usepackage{pstricks}               
\usepackage{pst-all}                
\usepackage{hyperref}

\hypersetup{
    colorlinks,
    citecolor=blue,
    filecolor=blue,
    linkcolor=blue,
    urlcolor=blue
}
\usepackage[a4paper,bindingoffset=0.2in,%
            left=1in,right=1in,top=1in,bottom=1in,%
            footskip=.25in]{geometry}

\theoremstyle{plain}
\newtheorem{theorem}{Theorem}[section]
\newtheorem{lemma}[theorem]{Lemma}
\newtheorem{proposition}[theorem]{Proposition}
\newtheorem{corollary}[theorem]{Corollary}

\theoremstyle{definition}
\newtheorem{definition}[theorem]{Definition}
\newtheorem{remark}[theorem]{Remark}

\newcommand{\di}{\mathrm{d}}

\newcommand{\Hpi}{H_{\pi}}

\begin{document}

\title{Poisson Structures of near-symplectic Manifolds and their Cohomology}

\author{Panagiotis Batakidis}

\address{Department of Mathematics
\\
Aristotle University of Thessaloniki
\\
 Thessaloniki 54124, Greece.}

\email{batakidis@math.auth.gr}

\author{Ram{\'o}n Vera}

\address{Department of Mathematics, KU Leuven, Celestijnenlaan 200B, Leuven 
B-3001, Belgium}

\email{ramon.vera@kuleuven.be, rvera.math@gmail.com}

\begin{abstract}
We connect Poisson and near-symplectic geometry by showing that there is a singular Poisson structure on a near-symplectic 4-manifold. The Poisson structure $\pi$ is defined on the tubular neighbourhood of the singular locus $Z_{\omega}$ of the 2-form $\omega$, it is of maximal rank 4 and it vanishes on a degeneracy set containing $Z_{\omega}$. We compute its smooth Poisson cohomology, which depends on the modular vector field and it is finite dimensional. We conclude with a discussion on the relation between the Poisson structure $\pi$ and the overtwisted contact structure associated to a near-symplectic 4-manifold.

\end{abstract}


\maketitle

\vspace{-1cm}

\section{Introduction}
It is well known that symplectic and Poisson structures are naturally related. A symplectic form on a smooth manifold determines a regular Poisson structure, whose symplectic leaf is the whole manifold.  Relaxing the non-degeneracy condition of a symplectic form leads to a closed 2-form that is symplectic away from its degeneracy locus, i.e it is singular with respect to non-degeneracy. It is then not automatic that there is an induced Poisson structure as in the symplectic case. In this work we study this problem in relation to a near-sympectic form, a type of such singular symplectic structure. This is a closed 2-form $\omega$ on a smooth $2n$--manifold $M$ that is positively non-degenerate outside a codimension-3 submanifold, where the rank of $\omega$ drops by 4.  If $M$ is 4-dimensional and closed, $Z_{\omega}$ is a collection of circles where $\omega$ vanishes. The idea of looking at near-symplectic forms goes back to Taubes in relationship to $J$--holomorphic curves, Seiberg-Witten, and Gromov invariants \cite{T98, T98SW, T99}. Near-symplectic forms have also been studied under the framework of self-dual harmonic forms vanishing on circles for a generic metric \cite{Ho04, LB97, T99}. The work of Auroux, Donaldson, and Katzarkov showed that a natural object associated to near-symplectic forms is a broken Lefschetz fibration \cite{ADK05}, a generalization of Donaldson's Lefschetz pencil. These 2-forms have been of interest also in  smooth 4-manifold theory \cite{GK, Ge18, L09} and contact topology due to their connection to overtwisted contact structures \cite{Ho04, GK07}.  Here, we take a distinct view by approaching them through Poisson geometry. We prove the existence of Poisson structures on near-symplectic manifolds, and characterize them in terms of their Poisson cohomology.

Poisson cohomology was introduced by Lichnerowicz in 1977 \cite{Lichnerowicz}. It is an important invariant of Poisson geometry, as it reveals features about deformations, normal forms, derivations, and other characteristics of a Poisson structure. In general it is hard to calculate, one of the reasons being that the complex used to define the cohomology spaces is elliptic only at the points where the Poisson bivector is non-degenerate. In many cases it is infinite-dimensional, and it is unknown for many types of Poisson structures. It is well known that if $\mathfrak{g}$ is a semisimple Lie algebra and $\mathfrak{g}^{*}$ its dual  equipped with the correspoinding Lie-Poisson structure, then by results of Lu \cite{Lu90}, Ginzburg and Weinstein \cite{GW92}, the Poisson and Lie algebra cohomologies with polynomial coefficients are related and in fact $H^{k}_{\pi}(\mathfrak{g}^{*}) = H_{\text{Lie}}^{k}(\mathfrak{g}) \otimes \text{Cas}(\mathfrak{g}^{*})$ (see for example \cite[Proposition 7.15]{LGPV13}).  However the linear Poisson structure constructed in this work is neither semi-simple, nor compact. 

Recently, Poisson cohomology has served as a valuable tool to understand certain singular Poisson structures. For example, it was essential in the work of Radko \cite{Radko} in order to classify topologically stable Poisson structures on smooth, compact, oriented, surfaces.  These structures were later generalized under the name of \textsl{log} or \textsl{b-symplectic} structures.  The Poisson cohomology of $b$-symplectic structures was determined in the work of Guillemin, Miranda, and Pires \cite{GMP}, and Marcut and Osorno-Torres \cite{MO14i, MO14ii}, while the Poisson cohomology of broken Lefschetz fibrations is computed in \cite{BV18}.

The main result of this paper is the following.
\begin{theorem}\label{thm:near-symplectic-Poisson}
Let $(M, \omega)$ be a closed, near-symplectic $4$--manifold.  Then there is a singular Poisson structure $\pi$ of maximal rank $4$ on the tubular neighborhood $U_{Z_\omega}$ of $Z_\omega$ such that the vanishing locus of $\pi$ contains $Z_\omega$. The smooth Poisson cohomology of $\pi$ is given in the following list, where $k$ denotes the total number of circles in $Z_\omega$:

\begin{equation*} 
\begin{split}
H^{0}_{\pi} (U_{Z_\omega}, \mathbb{R}) &\cong \mathbb{R}   \cong  \mathrm{span}\langle 1 \rangle, 
\\
H^{1}_{\pi}(U_{Z_\omega}, \mathbb{R}) & \cong \mathbb{R}^{2k}    \, \, \cong  \, \, \bigg[\bigoplus_{r=1}^{k} \mathrm{span} \left \langle  Y_{r}^{\Omega}(\pi) \, , \,   \partial_{r}^{L^1}  \right \rangle\bigg],
\\
H^{2}_{\pi} (U_{Z_\omega}, \mathbb{R})  &\cong \mathbb{R}^k  \cong\bigg[ \bigoplus_{r=1}^{k}  \mathrm{span} \left \langle Y_{r}^{\Omega}(\pi) \wedge   \partial_{r}^{L^1}  \right\rangle\bigg],
\\
H^{3}_{\pi} (U_{Z_\omega}, \mathbb{R}) & =0, 
\nonumber
\\
H^{4}_{\pi} (U_{Z_\omega}, \mathbb{R}) &=0. 
\nonumber
\end{split}
\end{equation*}
The generators $Y_{1}^{\Omega}(\pi), \dots, Y_{k}^{\Omega}(\pi)$ of $\Hpi^{1}(U_{Z_\omega},\mathbb{R})$ correspond to the modular vector field of $\pi$ at each component of the singular locus $Z_{\omega}$, and  $ \partial_{1}^{L^1}, \dots,  \partial_{k}^{L^1}$ are vector fields on the tubular neighbourhood of each component of $Z_{\omega}$. 
\end{theorem}
The proof of the existence part of Theorem \ref{thm:near-symplectic-Poisson} is Proposition \ref{prop:nearsymp}. Section \ref{sec:singular-structures} finishes with a remark regarding Poisson structures in higher dimensional near-symplectic manifolds (Proposition \ref{allM}).  

We then calculate the Poisson cohomology of the structure of Theorem \ref{thm:near-symplectic-Poisson} in Section \ref{near-positive cohomology}. We start by computing Poisson cohomology with formal coefficients in Proposition \ref{prop:near-pos-cohom}. This calculation is split in Lemmata \ref{H0} - \ref{H3} in Section \ref{smooth near-positive}. A key observation comes from the action of Hamiltonian vector fields on polynomial functions with respect to a certain notion of degree. The section finishes with Remark \ref{dq} about deformation quantization of this particular Poisson structure.  We then follow with Poisson cohomology with smooth coefficients in Proposition \ref{smooth np}. 

 In Section \ref{sec:contact} we discuss the relation between Poisson and contact structures on a near-symplectic 4-manifold in connection to Theorem \ref{thm:near-symplectic-Poisson}. It is known that there is an overtwisted contact structure on the boundary of the tubular neighbourhood of the singular locus of a near-symplectic form \cite{Ho04, GK}. We use this result to make some observations regarding the orbits of the Reeb vector field in relation to the modular class and the image of the contact form through the anchor map of the Poisson structure that we construct.  %

The Poisson structure studied here fits in the following degeneracy scheme on 4-manifolds. Let $M$ be a smooth oriented 4-manifold and  $\pi \in \Gamma(\Lambda^2 TM)$ a Poisson bivector.  In terms of distinct degeneracies in the rank of $\pi$, we have that at any point $p\in M$,  $\pi$ can have rank 4, 2, or 0 along symplectic leaves, so one has the following cases:

\begin{itemize}
\item[({\em i})] ${\rm Rank} (\pi_p) = 4$, where $\pi^2(p) \not= 0$,
\vspace{0.5mm}

\item[({\em ii})] ${\rm Rank}(\pi_p) = 2$, where $\pi^2(p) = 0$, but $\pi(p) \not= 0$,
\vspace{0.5mm}

\item[({\em iii})] ${\rm Rank}(\pi_p) = 0$.
\end{itemize}
\quad  Regular Poisson structures are those with constant rank on $M$. On one end of the spectrum we find symplectic manifolds, which determine a regular Poisson bivector satisfying condition ({\em i}) everywhere. On the other end, a trivial Poisson structure corresponds to case ({\em iii}). If a Poisson structure is singular, there can be a combination of the three cases in the list, at different points of the manifold. For instance, {\em log-symplectic} structures are those equipped with a Poisson bivector $\pi$ on an even dimensional manifold $M$ such that $\pi^n$ is transverse to the zero section in $\Lambda^{2n} TM$. In dimension 4, they capture cases ({\em i}) and ({\em ii}); the rank of $\pi$ is maximal except at a codimension-1 submanifold, where $\pi^2$ vanishes transversally. The Poisson structure that we consider here is an example for cases ({\em i}) and ({\em iii}). 

In \cite{BV18} we compute the Poisson cohomology of a broken Lefschetz fibration (bLf) using the associated Poisson structure constructed in \cite{GSV}. That Poisson structure is a combination of cases ({\em ii}) and ({\em iii}) in the previous list. With the Poisson cohomology computed in \cite{GMP} together with \cite{BV18} and this paper, one will then have available Poisson cohomology computations for large classes of singular Poisson structures on 4- manifolds.

\section{Preliminaries}
\subsection{Poisson Geometry and Cohomology}\label{Preliminary:Cohomology}
We recall some basic facts about Poisson geometry, referring  the reader to e.g. \cite{LGPV13} for details. Let $M$ be a smooth manifold and $C^\infty(M)$ be the sheaf of smooth $\mathbb{R}$-valued functions on $M$. A Poisson structure on $M$ is a Lie bracket $\{\cdot,\cdot\}$ on $C^\infty(M)$ obeying the Leibniz rule $\{fg,h\}=f\{g,h\}+g\{f,h\}$.  Let $\mathfrak{X}^{p}(M) = \Gamma(\Lambda^p TM)$ be the space of $p$- vector fields on $M$ and  $\left[ \cdot ,  \cdot \right]_{\mathrm{SN}} \colon \mathfrak{X}^{p}(M)  \times \mathfrak{X}^{q}(M) \rightarrow \mathfrak{X}^{p+q-1}(M)$ the Schouten-Nijenhuis bracket. A Poisson structure on $M$ can be equivalently described by a bivector field $\pi \in \mathfrak{X}^2(M)$, called {\em Poisson  bivector},  satisfying $\left[ \pi, \pi \right]_{\mathrm{SN}}=0$.  In local coordinates $\{x_1,\ldots,x_n\}$, a Poisson bivector is determined by an antisymmetric matrix $\pi^{i,j}$, written explicitly as $\pi=\sum_{1\leq i<j\leq n}\pi^{i,j}\partial_{i}\wedge\partial_{j}$. The pair $(M, \pi)$ is called a {\em Poisson manifold}. We henceforth assume a Poisson manifold $(M,\pi)$ and establish some notation.

For use in the computations of Section \ref{near-positive cohomology}, we fix the formula for $[\cdot,\cdot]_{\rm SN}$. Consider $\zeta_i=\partial_{x_i}$ as an odd variable, so that $\zeta_i\zeta_j=-\zeta_j\zeta_i$. A $p-$ vector field $P\in\mathfrak{X}^p(M)$ is written as $\displaystyle P=\sum_{i_1<\cdots<i_p}P_{i_1\cdots i_p}\zeta_{i_1}\cdots\zeta_{i_p}$, with $P_{i_1\cdots i_p}\in C^\infty(M)$. Then for $Q\in\mathfrak{X}^q(M)$, define
\begin{equation}\label{SN}
[P,Q]_{\rm SN}=\sum_i\partial_{\zeta_i}(P)\partial_{x_i}(Q)-(-1)^{(p-1)(q-1)}\partial_{\zeta_i}(Q)\partial_{x_i}(P).
\end{equation}

 \noindent where $\partial_{\zeta_{i_k}}\big(\zeta_{i_1}\cdots\zeta_{i_p}\big)=(-1)^{p-k}\zeta_{i_1}\cdots\widehat{\zeta_{i_k}}\cdots\zeta_{i_p}$.

Interior contraction with $\pi$ defines a vector bundle homomorphism, which on the spaces of sections reads $ \pi^{\sharp} \colon \Omega^1(M) \rightarrow \mathfrak{X}^1(M) $, and is given pointwise by 
$\pi^{\sharp}_{p} (\alpha_p) = \pi_p(\alpha_p,\cdot)$. It is called the {\em anchor map}. This map extends to a $C^\infty(M)$- linear homomorphism
\begin{equation}\label{sharp}
\wedge^\bullet\pi^\sharp:  \Omega^\bullet(M)\longrightarrow\mathfrak{X}^\bullet(M),
\end{equation}
which we denote again by $\pi^\sharp$ for simplicity. 

A vector field $X$ is said to be a {\em Poisson vector field}, if $\mathcal{L}_{X} \pi = 0$. Additionally, the vector field $X_f = \pi^{\sharp} (df) $ 
is called the {\em Hamiltonian vector field} of the Hamiltonian function $f\in C^{\infty}(M)$.  One can check directly that every Hamiltonian vector field is Poisson. 

Due to the Poisson condition on $\pi$, the operator 

\begin{equation}\label{cobound}
\di_\pi\colon \mathfrak{X}^{\bullet}(M) \rightarrow \mathfrak{X}^{\bullet + 1}(M),\;\;\;\;\;
X \mapsto \di_{\pi}(X) = \left[\pi,X \right]_{\mathrm{SN}}
\end{equation}

\noindent is a differential of the exterior algebra $\mathfrak{X}(M) = \oplus_{k} \mathfrak{X}^{k}(M)$ leading to the following.

\begin{definition} The pair $(\mathfrak{X}(M), \di_\pi)$  is called the {\em Lichnerowicz-Poisson cochain complex}, and 
$$H^{k}_{\pi}(M) := \frac{\ker\left( \di_{\pi}^k\colon \mathfrak{X}^{k}(M) \rightarrow \mathfrak{X}^{k+1}(M) \right) }{ \mathrm{Im}\left( \di_{\pi}^{k-1}\colon \mathfrak{X}^{k-1}(M) \rightarrow \mathfrak{X}^{k}(M) \right)},$$
are called the {\em Poisson cohomology spaces} of $(M, \pi)$. 
\end{definition}
\noindent For our purposes we recall the interpretation of the lower Poisson cohomology groups:
\begin{align}
H^{0}_{\pi}(M) &= \left\lbrace \textnormal{Casimir functions}  \right\rbrace  \quad 
\vspace{2mm}
\nonumber
\\
H^{1}_{\pi}(M) &= \frac{ \lbrace \textnormal{Poisson vector fields}\rbrace }{ \lbrace \textnormal{Hamiltonian vector fields} \rbrace}  
\nonumber
\\
H^{2}_{\pi}(M) &= \frac{ \lbrace \textnormal{infinitesimal deformations of $\pi$ }\rbrace }{ \lbrace \textnormal{trivial deformations of $\pi$} \rbrace}  
\nonumber
\\
H^{3}_{\pi}(M) &= \lbrace \textnormal{obstructions to formal deformations of $\pi$}  \rbrace.
\nonumber
\end{align}

The map (\ref{sharp}) is a chain map  and defines a homomorphism of graded Lie algebras 
\begin{equation}\label{chain}
\hat{\pi}^{\sharp}\colon H_{\text{dR}}^{\bullet}(M) \rightarrow H_{\pi}^{\bullet}(M).
\end{equation}
In general, $\hat{\pi}^{\sharp}$ is neither injective nor surjective, however if $(M,\omega)$ is symplectic with associated Poisson structure $\pi_\omega$,  its Poisson cohomology is known, as $\hat{\pi}_\omega^{\sharp}$ is an isomorphism: 
$$ H_{\text{dR}}^{\bullet}(M) \simeq H_{\pi_\omega}^{\bullet}(M),$$ and $[\pi^\sharp(\omega)] = \left[ \pi_{\omega} \right]$. 

The first cohomology group encompasses a distinctive object of a Poisson structure, the modular class. To define it, consider an orientable Poisson manifold with positive oriented volume form $\Omega$. The mapping 
\begin{align*}
Y^{\Omega}\colon C^{\infty}(M) &\rightarrow C^{\infty}(M)
\end{align*}
defined by 
$$\mathcal{L}_{X_{f}} \Omega = (Y^{\Omega} f) \Omega$$ is a Poisson vector field. The vector field $Y^{\Omega}$ is known as the {\em modular vector field} with respect to $\Omega$. For another choice $\Omega'=g\cdot \Omega$, $g\in C^\infty(M)$, the vector fields $Y^{\Omega'}$ and $Y^\Omega$ differ by a Hamiltonian vector field and thus there is a canonically defined Poisson cohomology class $\left[  Y^{\Omega} \right] $ called the \textsl{modular class} of $(M,\pi)$. One can check directly that $[Y^\Omega]=0$ if and only if $\Omega$ is invariant by the flows of all Hamiltonian vector fields. Modular vector fields and modular classes are defined for non-orientable Poisson manifolds using densities. Finally we recall that a Posson structure $\pi$ is said to be \textit{exact}  if the fundamental cohomology class $\left[ \pi \right]$ vanishes.

\subsection{Near-symplectic structures and Euler vector fields}\label{sec:near-symp}
\subsubsection{Near-symplectic Forms}

Since we are interested in the connection between Poisson and near-symplectic geometry, we briefly recall some facts about near-symplectic structures. We refer the reader to \cite{ADK05, P07, T98, T99, V16} and the references within for a detailed exposition on these structures.

Let $V$ be a 4-dimensional vector space. The wedge product $\wedge\colon \Lambda^2 V^{*} \otimes \Lambda^2 V^{*} \rightarrow \Lambda^4 V^{*}$ defines a quadratic form of signature $(3,3)$ on $\Lambda^2 V^{*}$. This decomposes the second degree of the exterior algebra as $\Lambda^2 V^{*} = \Lambda^2_{+} V^{*} \oplus \Lambda^2_{-} V^{*}$, where $\Lambda^2_{+} V^{*} = \lbrace \alpha\in \Lambda^2 V^{*} \mid \alpha\wedge \alpha \geq 0 \rbrace$ and $\Lambda^2_{-} V^{*}= \lbrace  \alpha\in \Lambda^2 V^{*} \mid \alpha \wedge \alpha \leq 0 \rbrace$.  In  coordinates $(\theta, x_1, x_2, x_3)$ on $V$, one can write a basis for each one of these spaces

\begin{align}\label{BasisSD}
&\Lambda_{+}^2 V^*=\mathbb{R}\langle\beta_1,\beta_2,\beta_3\rangle\;\;\text{with}                                  &  \qquad  &     \Lambda^{2}_{-}V^*=\mathbb{R}\langle\beta_4,\beta_5,\beta_6\rangle\;\;\text{with}                                                           
\\
&\beta_1 = d\theta\wedge dx_1 + dx_2\wedge dx_3,   &  \qquad  &   \beta_4 = d\theta\wedge dx_1 - dx_2\wedge dx_3 ,        
\nonumber \\
&\beta_2 = d\theta\wedge dx_2 - dx_1\wedge dx_3,   &  \qquad  &   \beta_5 = d\theta\wedge dx_2 + dx_1\wedge dx_3,         
\nonumber \\
&\beta_3 = d\theta\wedge dx_3 + dx_1\wedge dx_2 ,  &  \qquad  &   \beta_6 = d\theta \wedge dx_3 - dx_1\wedge dx_2 .        
\nonumber 
\end{align}

Let $M$ be a smooth, oriented 4-manifold.  Consider a 2-form $\omega\in \Omega^2(M)$ with the property of being {\em near-positive} everywhere, that is $\omega^2 \geq 0$.  Such a form can have rank 0, 2, or 4 at any point.  Genericity results (see Theorem \ref{Taubes98} below) show us that it is of interest to look at 2-forms that only have rank 0 or 4 at any point.   Let $D \omega \colon TM \rightarrow \Lambda^2 T^{*}M$ be the derivative of $\omega$ on tangent spaces (not exterior differential). Since $\omega$ is assumed to be near-positive, the image $D\omega$ can be at most of dimension 4. By setting ${\rm Rank} (D\omega) = 3$, one obtains an identification of the image of $D\omega$ with the positive bundle of self-dual forms.

\begin{definition}
A {\em near-symplectic form} is a closed 2-form $\omega\in \Omega^2(M)$ such that $\omega^2 \geq 0$ and at every point $p\in M$, either
\begin{itemize}
\item[(i)] $\omega_p$ is symplectic, or 

\item[(ii)] $\omega_p = 0$ and ${\rm Rank}(D\omega_p) = 3$. 
\end{itemize}
Its {\em singular locus}, $Z_{\omega} = \lbrace p\in M \, \mid \, \omega_p=0\rbrace$ is a 1-submanifold of $M$. 
\end{definition}

 It is possible to modify or reduce the number of components of the zero locus, but it has been shown that $Z_{\omega}$ is always non-empty unless the underlying manifold is symplectic \cite[Section 5]{T98SW}.

We recall the local expression of a near-sympletic form. Keeping the notation of local coordinates $(\theta, x_1, x_2, x_3)$, the basis elements (\ref{BasisSD}) will be used with the same letters to write local sections of $\Lambda_{+}^2 T^*M$.  A Darboux-type theorem \cite[Lem. 3.1]{P07}, \cite[Cor. 3.1]{V16} for near-symplectic forms tells us that on the tubular neighbourhood of $Z_{\omega}$ a near-symplectic form has the formal normal form  
\begin{equation}\label{omega:Darboux}
\omega =   x_1 \beta_1  - 2  x_2 \beta_2  + x_3 \beta_3.
\end{equation}
With respect to this model, $Z_{\omega}$ is given by the submanifold $\lbrace x_1 = x_2 = x_3 = 0 \rbrace$.

\subsubsection{Properties of near-symplectic forms}

Near-symplectic forms are related to self-dual harmonic 2-forms for some Riemannian metric.  This equivalent formulation appears in the work of different authors \cite{LB97, Ho04, T98, T98SW, T99}.  The relation between these geometric objects is described through the following statement.

\begin{theorem}\cite[Thm. 4]{T98} \cite[Prop. 1]{ADK05} \label{Taubes98}\label{PropADK}
Let $M$ be a smooth, oriented 4-manifold. For a near-symplectic form $\omega$ on $M$, there is a Riemannian metric $g$ on $M$ such that $\omega$ is self-dual and harmonic with respect to $g$. Conversely, if $M$ is compact and $b^{2}_{+}(M) \geq 1$, then for a generic Riemannian metric $g$ there is a closed, self-dual harmonic form $\omega$, that vanishes transversally as a section of $\Lambda^{2}_{+} T^*M$ and defines a near-symplectic structure. The zero set of $\omega$ is a finite, disjoint union of embedded circles.
\end{theorem}

\begin{remark}\label{rem-nearsymp but not symp}
There are smooth 4-manifolds that can have a near-symplectic form but not a symplectic one. For instance, it is well-known that if $M$ is a smooth, closed, 4-manifold admitting a symplectic form then $(1-b_1 + b_{2}^{+})$ must be even.  Hence, if $M$ is simply connected, then $b_{2}^{+}$ must be odd.  Nevertheless, any $M$ with $b_{2}^{+} \geq 1$ admits a near-symplectic form. 
\end{remark}

\begin{remark}\label{rem-splitting-nz}  
A near-symplectic form has the property of splitting the normal bundle $NZ_{\omega}$ of its singular locus $Z_{\omega}\subset M$ into two subbundles, a rank-1 bundle $L^1_{-}$ and a rank-2 bundle $L^2_{+}$. To see this, one uses the geometric information of the near-symplectic structure to construct a self-adjoint, trace-free automorphism $F\colon NZ_{\omega}\rightarrow NZ_{\omega}$. Its representative matrix is symmetric, traceless, and has three eigenvalues, two positive and one negative (for more details see \cite[sections 3 \& 4]{Ho04}, \cite[sec. 2.3]{P07}, \cite[sec. 2c]{T99}, \cite[sec. 4B]{V16}). The negative and positive eigensubspaces draw the corresponding bundles $L^1_{-}$ and $L^2_{+}$ . These properties are independent on the choice of the metric $g$. There can be many conformal classes $[g]$ for which $\omega$ is self-dual, yet they are all the same along $Z_{\omega}$ because $D \omega$ identifies the normal bundle $NZ_{\omega}$ with  $\Lambda^2_{+} T^*M$ at each point of $Z_{\omega}$.  Therefore, a near-symplectic form $\omega$ determines a canonical embedding of the intrinsic normal bundle $NZ_{\omega}$ as a subbundle of $TM|_{Z_\omega}$ complementary to $TZ_\omega$ \cite{P07, T99}.  This is summarized for later use in the following Lemma.
\end{remark}

\begin{lemma}\cite{P07, T99, V16}
Let $(M, \omega)$ be a near-symplectic manifold with singular locus $Z_{\omega}$. The normal bundle $NZ_{\omega}$ of $Z_{\omega}$ splits into a line bundle $L^{1}_{-}$ and a rank 2 bundle $L^{2}_{+}$, i.e $NZ_{\omega} \simeq L^{1}_{-} \oplus L^{2}_{+}$. 
\end{lemma}

\subsection{Euler-like vector fields and Tubular Neighbourhoods}
 In this section we recall some notions on linear approximations, Euler-like vector fields, and tubular neighbourhoods based on \cite{BLM}.  Let $Z\subset M$ be a smooth submanifold and denote by $NZ= \nu(M,Z)= TM|_Z \slash TZ$ the normal bundle of $Z$. Let also  $p\colon \nu(M,Z) \rightarrow Z$, $i\colon Z \rightarrow M$ be the projection and inclusion maps. 

For a vector bundle $F\rightarrow Z$, the normal bundle relative to the zero section is  $\nu(F,Z) = F$. The normal bundle of $TM$ relative to $TZ$ is canonically isomorphic to the tangent bundle of the normal bundle. In particular, the normal and the tangent functors commute, and there is a canonical isomorphism $\nu(TM, TZ) \cong T_{\nu}(M,Z)$. Let $N_1 \subset M_1$ and $N_2 \subset M_2$ be submanifolds of $M_1$ and $M_2$.  A smooth map of pairs $\psi\colon (M_1, N_1) \rightarrow (M_2, N_2)$ taking $M_1$ to $M_2$, and $N_1$ to $N_2$, induces a map on normal bundles $\nu(\psi) \colon \nu(M_1, N_1) \rightarrow \nu(M_2, N_2)$.  For instance, take a vector field $X\in \mathfrak{X}(M)$ tangent to a submanifold $Z$. View $X$ as a section $M\rightarrow TM$. The condition of $X$ being tangent to $Z$ means that it takes $Z$ to the submanifold $TZ$, i.e. it defines a map of pairs $X\colon (M,Z) \rightarrow (TM, TZ)$. Applying the normal functor, one obtains a map $\nu(X) \colon \nu(M,Z) \rightarrow \nu(TM, TZ) = T\nu(M,Z)$.  In this way, from a vector field tangent to $Z$, one obtains a vector field on the normal bundle of $TM$ relative to $TZ$, called the {\em linear approximation}.  Linear approximation is a coordinate-free way of defining a tensor field, including Poisson bivectors and other multivector fields.  

\begin{definition}\cite[Def. 2.6]{BLM}
Let $Z\subset M$ be a submanifold and $\mathcal{E} \in \mathfrak{X}(\nu (M,Z))$ an Euler vector field. A vector field $R\in \mathfrak{X}(M)$ is called {\em Euler-like} if $R$ is complete, $R|_{Z} = 0$, with linear approximation $\nu(R) = \mathcal{E}$. 
\end{definition}

\noindent Linear approximations serve in the following definition of tubular neighbourhoods. 

\begin{definition}\cite[Def. 2.3]{BLM}
A {\em tubular neighbourhood embedding} for $Z\subset M$ is an embedding of the normal bundle $\psi\colon \big(\nu (M, Z), Z \big) \rightarrow (M, Z)$ such that: {\em (i)} it takes the zero section of $\nu(M,Z)$ to $Z$, and {\em (ii)} its linear approximation is the identity map, i.e. $\nu(\psi)= \textnormal{id}$. 
\end{definition}

There is a direct connection between Euler-like vector fields and tubular neighbourhood embeddings. If $\mathcal{E}$ is the Euler vector field on the normal bundle, then any tubular neighbourhood embedding carries $\mathcal{E}$ to an Euler-like vector field defined in a neighborhood of $Z$ in $M$. 

\begin{proposition}\cite[Prop. 2.7]{BLM}\label{PropBLM}
Let $Z\subset M$ be a submanifold and $\mathcal{E} \in \mathfrak{X}(\nu(M,Z))$ an Euler vector field. Any $R\in \mathfrak{X}(M)$ Euler-like along $Z$,  determines a unique tubular neighbourhood embedding $\psi\colon \nu(M,Z) \rightarrow M$ with 
$$\psi_{*}\mathcal{E} = R.$$
\end{proposition}

\noindent In particular, Euler-like vector fields are always {\em linearizable} \cite[Lemma 2.4]{BLM}. 

\section{Induced singular Poisson structure}\label{sec:singular-structures} 
In this section we construct a Poisson structure in a near-symplectic manifold.

\subsection{Poisson structures in near-symplectic 4-manifolds}\label{sec:near-symp-Pois}

\begin{proposition}\label{prop:nearsymp}
Let $(M, \omega)$ be a closed near-symplectic 4-manifold with singular locus $Z_{\omega}$. Denote by $U_Z\subset M$ a tubular neighbourhood of $Z_{\omega}$.   There is a singular Poisson structure $\pi_U$ of generic rank 4 on $U_Z \subset M$ with the following characteristics:  The degeneracy locus $D_\pi \subset U_Z$ of $\pi_U$ is a 2-dimensional surface containing $Z_\omega$ and \newline
i) $\pi_U|_p^2 > 0$ for all $p\in  U_Z \setminus  D_{\pi}$\newline
ii) $\pi_U|_p = 0$  for all $p\in D_{\pi}$. 

\end{proposition}

\begin{proof} Recall that given a near-symplectic form, the normal bundle splits into $NZ_{\omega} = L^1 \oplus L^2$, a rank 1-bundle $L^1$ and a rank 2-bundle $L^2$.   We use this splitting property induced from $\omega$ to construct a Poisson structure on the tubular neighbourhood of $Z_{\omega}$. 

Let  $\mathcal{E}_{-}, \mathcal{E}_{+}$ be Euler vector fields on $L^1_{-}$ and $L^2_{+}$ respectively. In bundle coordinates, with $(x_2)$ on $L^1_{-}$, and $(x_1, x_3)$ on $L^2_{+}$  they are expressed as 

$$\mathcal{E}_{-} = x_2 \frac{\partial}{\partial x_2},
\qquad 
\mathcal{E}_{+} = x_1 \frac{\partial}{\partial x_1} + x_3 \frac{\partial}{\partial x_3}.
$$
In particular, one can read that $\mathcal{E}_{-}|_{L^{2}_+} = 0,\; \mathcal{E}_{+}|_{L^{1}_{-}} = 0$, , and  $\mathcal{E}_{\pm}|_{Z_{\omega}} = 0$ at the zero section.

Recall that in dimension 4 the singular locus $Z_{\omega}$ consists of a collection of embedded circles. For purposes of clarity we will work over one connected component, i.e.  one circle $Z_{\omega} = S^1$. Let $X \in \mathfrak{X}(S^1)$ be the unit tangent vector field so that $X = \frac{\partial}{\partial \theta}$. Let $R \in \mathfrak{X}(M)$ be an Euler-like vector field along $Z_{\omega}$.  By proposition \ref{PropBLM} a unique tubular neighbourhood embedding $\psi \colon NZ_{\omega} \to M$ is determined by $R$ with $\psi_{*} \mathcal{E} = R$.  

Denote by $U_Z = \psi(NZ_{\omega})$ the tubular neighbourhood of $Z_{\omega}$ in $M$.   Define on $U_Z$ the following bivector field 
\begin{equation}
\eta:= X\wedge \psi_{*} \mathcal{E}_+ = X \wedge R_{+} ,
\end{equation}
where $R_+\in\mathfrak{X}(M)$ is Euler-like on $U_Z$ with $\psi_{*}(\mathcal{E}_{+}) = R_{+}$. 

The previous $\eta$ constitutes one part of the bivector we want to construct. Recall that given a near-symplectic form on a closed 4-manifold, there is a metric $g$ such that $\omega$ is self-dual and vanishes on a collection of circles. 

Let $\ast$ be the Hodge operator with respect to this $g$ such that $\ast \omega=\omega$. Using the orientation given by the volume form $\omega^2$, one can define a Hodge duality isomorphism from the exterior algebra of the cotangent bundle to the one of the tangent bundle, thus obtaining a transformation of bivector fields, $\Lambda^2 TM \rightarrow \Lambda^2 TM $.  This transformation acts again as Hodge operator, and it is defined with respect to the volume form and a metric that makes $\omega$ self-dual.  Hence, by slight abuse of notation we denote it again by $\ast \colon \Lambda^2 TM \rightarrow \Lambda^2 TM$. The construction is independent on the particular choice of $\omega$ and $g$, since given any near-symplectic form, we can find a Riemannian metric $g$ such that $\omega$ is a self-dual harmonic 2-form vanishing on a 1-submanifold of $M$ (see Thm. \ref{Taubes98}).

Consider the following bivector field on $U_Z$
\begin{equation}\label{pi:U}
\pi_U = \eta + \ast \eta =  X\wedge R_{+} + \ast (X\wedge R_{+}).
\end{equation}
For a sufficiently small neighbourhood around $Z_{\omega}$, the linear model of $\pi_U$ is given by 
\begin{equation}\label{pi:near-pos}
\pi_U = x_1 \left( \frac{\partial}{\partial \theta} \wedge \frac{\partial}{\partial x_1}  + \frac{\partial}{\partial x_2} \wedge \frac{\partial}{\partial x_3} \right)  + x_3 \left( \frac{\partial}{\partial \theta} \wedge \frac{\partial}{\partial x_3}  + \frac{\partial}{\partial x_1} \wedge \frac{\partial}{\partial x_2}
\right).
\end{equation}
This bivector vanishes on $\psi(L_{-}^{1})$, which includes the singular locus of $\omega$. A calculation shows that it satisfies the Poisson condition $\left[ \pi_U, \pi_U \right]_{\mathrm{SN}}=0$, and  $\pi_U^2 \geq 0$.
\end{proof}

Recall that the decomposition $NZ_{\omega} \simeq L^{1}_{-} \oplus L^2_{+}$ has two possible splittings, since the line bundle $L^{1}_{-}$ can be oriented or not. We mention this because the previous exposition finishing in Prop. \ref{prop:nearsymp} dealt with the case when $L^{1}_{-}$ is oriented.  Now we say a word about this splitting and conclude by addressing the non-oriented case in Lemma \ref{lemma:involution}.
 
From a topological perspective, since $U_Z \simeq S^1 \times D^3$ one can regard the tubular neighbourhood of a component of $Z_{\omega}$ to be the total space $S^1 \times D^3$ of a disk bundle over a circle given by a projection map, $S^1\times D^3 \rightarrow S^1$. In general, $D^{n-1}$-bundles over $S^1$, with a splitting $D^k\times D^{n-1}$ for $0 < k < n-1$, are classified by homotopy classes from $S^1$ into the Grassmanian $G(n-1, k)$. Due to the fact that $\pi_1 (G(n-1, k)) = \mathbb{Z}_2$, one finds two possible splittings up to isotopy. In the dimensions we are working, this can be observed by looking at $\pi_1(\mathbb{RP}^2) = \mathbb{Z} \slash \mathbb{Z}_2$.  Hence, this disk bundle preserves the decomposition of $NZ_{\omega}$ and splits into a $D^2$-bundle and a $D^1$-bundle over $S^1$. 

Thus it just remains to be checked that the model $\pi_U$ is also valid on the non-trivial splitting of $S^1 \times D^3$ for the non-oriented case. This is shown in the next lemma.

\begin{lemma}\label{lemma:involution}
The bivector field $\pi_U$ in \eqref{pi:near-pos} is Poisson on the two homotopy classes of splittings of $S^1 \times D^3 \rightarrow S^1$ over each component of $Z_{\omega}$. 
\end{lemma}

\begin{proof} The oriented case has already been shown through the previous proposition. 
The non-oriented model is given by the quotient of $S^1 \times D^3$ by an involution $\iota$ reversing the orientation on both summands of the splitting \cite{Ho04}. Explicitly,

\begin{align}\label{involution}
\iota\colon S^1 \times D^3 &\rightarrow S^1 \times D^3
\\
(\theta, x_1, x_2, x_3) &\mapsto (\theta+ \pi, -x_1, x_2, -x_3).
\nonumber
\end{align}

We just need to check that if the normal bundle is non-orientable, the local model \eqref{pi:near-pos} still provides a Poisson structure.  From the action of $\iota$ we obtain 

\begin{equation*}
\iota_{*}  \left ( \frac{\partial}{\partial \theta} \right ) =\frac{\partial}{\partial \theta}, 
\quad 
\iota_{*}  \left ( \frac{\partial}{\partial x_1} \right ) =-\frac{\partial}{\partial x_1}, 
\quad 
\iota_{*}  \left ( \frac{\partial}{\partial x_2} \right )=\frac{\partial}{\partial x_2}, 
\quad 
\iota_{*}  \left ( \frac{\partial}{\partial x_3} \right )=- \frac{\partial}{\partial x_3}.
\end{equation*}
Thus, $\iota_{*}\pi=\pi$ and the involution $\iota$ is a Poisson map for $\pi$.
\end{proof}

\begin{remark}
The construction of the Poisson structure $\pi_U$ from a near-symplectic form $\omega$ is not functorial. Furthermore,  the Poisson structure constructed here is not $(\omega)^{-1}$, that is, it is not the bivector associated with the symplectic form on the regular part.  There could be other Poisson structures in a near-symplectic manifold besides the one we construct, perhaps even at a global level. To this end, we tried to construct a deformation path of appropriate near-symplectic forms following Luttinger and Simpson \cite{LS96} but we run into different obstructions. In addition, given $\pi_U$ as presented above, one would not be able to reconstruct a near-symplectic form $\omega$ without making additional choices. 

There are nevertheless some geometrical features coming from the near-symplectic structure that serve in the construction of $\pi_U$. The Poisson bivector $\pi_U$ does depend on having a particular rank 3 vector bundle over a circle with a particular splitting into a rank 1 vector bundle and a rank 2 vector bundle.  This feature is guaranteed by a near-symplectic form where the normal bundle offers such a splitting. Additionally, the Poisson bivector is self-dual with respect to a metric g, and this metric is naturally associated with a near-symplectic form (see Theorem \ref{PropADK}). As noticed before, near-symplectic forms are generic as closed, self-dual forms with respect to Riemannian metrics. 

\end{remark}

\subsection{Remark on Higher Dimensions}
There is a notion of near-symplectic forms in dimension $2n$ (see \cite{V16} for more details).  On an oriented $2n$-dimensional manifold $M$, a {\em near-symplectic form} is a closed 2-form such that at every point either
\begin{itemize}
\item[(i)] $\omega_p^n>0$, i.e. $\omega_p$ is symplectic, or 

\item[(ii)] $\omega_p^{n-1}=0$ transversally along a codimension-3 submanifold of $M$. 
\end{itemize}

At the degeneracy points, such a form has a 4-dimensional kernel $K_p = \lbrace v\in T_pM \, \mid \, \omega_p(v, \cdot) = 0 \rbrace$.  The collection of fibrewise kernels constitutes the kernel $K: = \ker (\omega) \subset TM$ of the 2-form. 
At every degenerate point $p\in M$, the splitting of the wedge square holds: $\Lambda^2 K^{*} = \Lambda^2_{+} K^{*} \oplus \Lambda^2_{-} K^{*}$.  The bundle $\Lambda^2 K^*$ is a subbundle of $\Lambda^2 T^{*}M$ over $M$ and $\Lambda_{+}^2K^*$, $\Lambda_{-}^2K^*$ are rank 3 bundles. 

On $Z_{\omega}$, the 2-form $\omega_Z = i^{*}\omega$ is a closed 2-form of constant rank $2n-4$, thus it defines a presymplectic structure on $Z_{\omega}$. The corresponding local Darboux form is given by 
\begin{equation}\label{omega:Darboux}
\omega =  \omega_Z + x_1 \beta_1  - 2  x_2 \beta_2  + x_3 \beta_3, 
\end{equation}
where 
$\omega_Z = \sum_{i=1}^{n-2} dq_i \wedge dp_i$ \cite[Lem. 3.1]{P07}, \cite[Cor. 3.1]{V16}.

\begin{proposition}\label{allM}
Let $(M, \omega)$ be a near-symplectic manifold of $\dim(M) =2n$ with singular locus being a symplectic mapping torus $Z_{\omega} = \big((Q, \omega_Q) \times [0,1] \big) \slash \stackrel{\phi}{\sim}$. 
There is a Poisson structure on the tubular neighbourhood of $Z_{\omega}$ in $M$ such that $\pi^{n-1}$ vanishes on $Z_{\omega}$. 
\end{proposition}

\begin{proof}
We extend the construction of Proposition \ref{prop:nearsymp} by first defining a Poisson bivector $\pi_U$ on the tubular neighbourhood $U_Z$ of $Z_\omega$ as in equation \eqref{pi:U}, and then adding a symplectic Poisson structure on $Z_{\omega}$. Since $Z_{\omega}$ fibres over $S^1$ and $\varepsilon = \ker(\omega_Z)$ is an integrable line bundle, there is a non-vanishing section $X\in \Gamma (\Lambda^1 \varepsilon)$.  The kernel $K:= \varepsilon \oplus NZ_{\omega} \subset TM$ of $\omega$ splits as $K = \varepsilon \oplus L_{-}^{1} \oplus L_{+}^{2}$ due to the splitting of $NZ_{\omega}$ (Remark \ref {rem-splitting-nz}).  Consider the Euler vector field $\mathcal{E}_{+}$ on $L^2_{+}$. By definition the kernel $K\subset TM$ is a rank 4 bundle, and since $Z_{\omega}$ is a mapping torus, we can look at self-dual forms on $\Lambda^2K^{*}$ vanishing on circles.  Fix a metric $g_K$ on $K$ such that $\omega$ is self-dual with respect to $g_K$ on $\Lambda^2 K^*$.   Using the orientation given by the volume form $\omega^n$, we can obtain a transformation of bivector fields, $\ast_{g_{K}} \colon \Lambda^2 K \rightarrow \Lambda^2 K$.

The 2-form $\omega_Q$ descends to the quotient and is symplectic on $Z_{\omega}$. Moreover, the horizontal distribution $\mathcal{H}\subset TZ_{\omega}$ is involutive, thus the bivector field $\pi_Z := \omega_Q^{-1}$ defines a symplectic Poisson structure on $Z_{\omega}$.  This Poisson bivector has the property that $\pi^{n-2}_{Z} \not= 0$,  $\pi_{Z}^{n-1} = 0$.  On the tubular neighbourhood $U_Z$ define the bivector field 

\begin{equation}
\pi_U = X\wedge R_{+} + \ast_{g_K} (X\wedge R_{+}) + \pi_Z.
\end{equation} 
A local calculation shows that $\left[ \pi_U, \pi_U \right]_{\mathrm{SN}} = 0$, because it follows the same computation as in dimension 4, except for a symplectic Poisson bivector that is added to it. 
\end{proof}

\begin{remark}
The Poisson structure on a near-symplectic manifold that we constructed in Proposition \ref{allM}, belongs to the class of {\em almost regular} Poisson structures \cite{AZ} since it is generically symplectic. Almost regular Poisson structures include regular Poisson and log-symplectic structures  among others.  The Poisson structure induced by a near-symplectic form is neither regular, nor log-symplectic. 
\end{remark}

%
%

\section{Poisson Cohomology on 4-manifolds}\label{near-positive cohomology}

In this section we compute the Poisson cohomology with smooth coefficients of the Poisson structure constructed in the previous section.  Section \ref{notation d} contains the formulas of the coboundary operator $\di$, and Section \ref{smooth near-positive} computes the cohomology groups with formal coefficients before extending those computations in smooth cohomology. Our results show that the Poisson cohomology spaces vanish except for vector and bivector fields with constant coefficients. In particular, the modular field $\partial_0$ has a nontrivial cohomology class, while $[\pi]=0$.

\subsection{The Poisson coboundary operator}\label{notation d}
We start by writing down the equations of the Poisson coboundary operator (\ref{cobound}). To simplify the notation relabel the variable $\theta$ as $x_0$ and set $\partial_i:= \frac{\partial}{\partial x_i}$, so the local model of such Poisson bivector on the tubular neighbourhood $U_Z$ is 
\begin{equation}\label{eq:Near-Positive}
\pi=x_1(\partial_0\wedge\partial_1+\partial_2\wedge\partial_3)+x_3
(\partial_0\wedge\partial_3+\partial_1\wedge\partial_2). 
\end{equation}
The Hamiltonian vector fields of the coordinate functions for the Poisson structure (\ref{eq:Near-Positive}) are
\begin{align}
\pi^\sharp(\di x_0)&=x_1\partial_1+x_3\partial_3,
\label{eq:pisharp-dx0-nearpos}
\\
\pi^\sharp(\di x_1)&=-x_1\partial_0+x_3\partial_2,
\label{eq:pisharp-dx1-nearpos}
\\
\pi^\sharp(\di x_2)&=-x_3\partial_1+x_1\partial_3,
\label{eq:pisharp-dx2-nearpos}
\\
\pi^\sharp(\di x_3)&=-x_3\partial_0-x_1\partial_2
\label{eq:pisharp-dx3-nearpos}
\end{align}
Setting $X_k:=\pi^\sharp(\di x_k)$ one may rewrite the Poisson bivector (\ref{eq:Near-Positive}) as 
\[\pi=\frac{1}{2}\sum_{k=0}^3\partial_k\wedge X_k.\]

\noindent Recall that $\di=[\pi,\cdot]_{SN}: \mathfrak{X}^{\bullet}(\mathbb{R}^4)\rightarrow\mathfrak{X}^{\bullet+1}(\mathbb{R}^4)$  is the Poisson coboundary operator. For $f\in C^\infty(\mathbb{R}^4)$, it is then
\begin{align}\label{d^0np}
\di^0(f)=&[\pi,f]_{SN}\stackrel{(\ref{SN})}{=}\sum_{i=0}^3\partial_{\zeta_i}(\pi)\partial_i(g)+0
\\
\nonumber
=&   -     \frac{1}{2}\sum_{i=0}^3\sum_{s=0}^3\{x_i,x_s\}\partial_s\wedge\partial_i(f)    +  \frac{1}{2}\sum_{i=0}^3\sum_{k=0}^3\{x_k,x_i\}\partial_k\wedge \partial_i(f)
\\
\nonumber
=& 
 -    \sum_{i,k=0}^3\partial_i(f)\{x_i,x_k\}\partial_k=
\sum_{k=0}^3X_k(f)\partial_k,
\\
\nonumber
\end{align}
\noindent where we used the expression $X_k=\sum_{i=0}^3\{x_k,x_i\}\partial_i$
for the hamiltonian vector fields \eqref{eq:pisharp-dx0-nearpos}-\eqref{eq:pisharp-dx3-nearpos}.

\noindent Let $Y=\sum_{k=0}^3f_k\partial_k\in\mathfrak{X}^1$, and $s$ be the index completing the triplet $\{1,2,3\}$ once $i<j$ are chosen. Then setting $\partial_{ij}:=\partial_i\wedge\partial_j$ for $i<j$,
\begin{align}\label{d^1np}
\di^1(Y)=& \sum_{k=1}^3\big[X_0(f_k)-X_k(f_0)-\frac{1-(-1)^k}{2}f_k\big]\partial_{0k}
\nonumber
\\
+& \sum_{i<j=1}^3\big[X_i(f_j)-X_j(f_i)-\frac{1-(-1)^{i+j}}{2}f_s\big]\partial_{ij}.
\end{align}
\medskip
For clarity in our upcoming computations, we write $\di^1(Y)$ in its expanded form
\begin{align}\label{d^1exp}
\di^1(Y)=& \big[X_0(f_1)-X_1(f_0)-f_1\big]\partial_{01}
\nonumber
\\
+& \big[X_0(f_2)-X_2(f_0)\big]\partial_{02}
\nonumber
\\
+& \big[X_0(f_3)-X_3(f_0)-f_3\big]\partial_{03}
\nonumber
\\
+& \big[X_1(f_2)-X_2(f_1)-f_3\big]\partial_{12}
\nonumber
\\
+& \big[X_1(f_3)-X_3(f_1)\big]\partial_{13}
\nonumber
\\
+& \big[X_2(f_3)-X_3(f_2)-f_1\big]\partial_{23},
\end{align}
\medskip

\noindent  Denote an arbitrary bivector field as $\displaystyle W=\sum_{0\leq i<j\leq 3}f_{ij}\partial_{ij}\in\mathfrak{X}^2$. Furthermore, we set $\partial_{ijk}:=\partial_i\wedge\partial_j\wedge\partial_k$ for $i<j<k$. Then
\begin{align}\label{d^2np}
\di^2(W)=&\big[X_0(f_{12})-X_1(f_{02})+X_2(f_{01}) - f_{12} + f_{03}\big]\partial_{012}
\nonumber
\\
+ & \big[X_0(f_{13})-X_1(f_{03})+X_3(f_{01}) - 2f_{13}\big]\partial_{013}
\nonumber
\\
+&\big[X_0(f_{23})-X_2(f_{03})+X_3(f_{02}) + f_{01} - f_{23}\big]\partial_{023}
\nonumber
\\
+&\big[X_1(f_{23})-X_2(f_{13})+X_3(f_{12})\big]\partial_{123}.
\end{align}
\medskip

\noindent Finally, let $\displaystyle Z=\sum_{0\leq i<j<k\leq 3}f_{ijk}\partial_{ijk}\in\mathfrak{X}^3$ be an arbitrary 3-vector field. Then

\begin{equation}\label{d^3np}
\di^3(Z)= \big[X_0(f_{123})-X_1(f_{023})+X_2(f_{013})-X_3(f_{012})-2f_{123}\big]\partial_{0123}.
\end{equation}

\subsection{Smooth cohomology}\label{smooth near-positive}

\subsubsection{Preliminaries}
We start by setting some notation. Let 

\begin{tabular}{l}
$V_i=\mathbb{R}_i[x_0,x_1,x_2,x_3]$  be the space of homogeneous polynomials of degree $i\in\mathbb{N}_0$,
\\
$r_i:=\dim(V_i)$,
\\
$\mathfrak{X}^m_i$ the space of $m$-vector fields on the tubular neighborhood $U_Z$, whose \\
coefficients are elements of $V_i$,
\\
$V_\textrm{formal}=\mathbb{R}[[x_0,x_1,x_2,x_3]]$,
\\
$\mathfrak{X}^m_\text{formal}$ be the space of $m-$ vector fields on $U_Z$ with coefficients from $V_\textrm{formal}$.
\end{tabular} 

The restriction of $\di^m$ to $\mathfrak{X}^m_{\text{formal}}$ is denoted with the same letter, and since $\pi$ is linear, it can be further decomposed as 
\begin{equation}\label{split}
\di^m=\sum_{i=0}^\infty \di^m_i,\;\;\;\text{with}\;\;\di^m_i:=\di^m|_{\mathfrak{X}^m_i}: \mathfrak{X}^{m}_i\rightarrow\mathfrak{X}^{m+1}_i.
\end{equation}

With the notation for polyvector fields and their coefficient functions as in equations (\ref{d^0np}),  (\ref{d^1np}),  (\ref{d^2np}), and (\ref{d^3np}), the operators $\di^\bullet_i$ are identified as the maps in the following sequence of spaces representing the coefficients in the complex $(\mathfrak{X}^\bullet_i, \di^\bullet_i)$:

\begin{equation}\label{seq}
0\longrightarrow V_i\stackrel{\di_i^0}{\longrightarrow} V_i^{\otimes 4}\stackrel{\di_i^1}{\longrightarrow} V_i^{\otimes 6}\stackrel{\di_i^2}{\longrightarrow} V_i^{\otimes 4}\stackrel{\di_i^3}{\longrightarrow} V_i\stackrel{\di_i^{4}=0}{\longrightarrow} 0\longrightarrow\cdots
\end{equation}
and more precisely, dropping $\di_i^m,\;\forall m\geq4$,

\begin{equation}\label{seqf}
 f\stackrel{\di^0_i}{\longrightarrow} ( f_{0},f_{1},f_{2},f_{3}) \stackrel{\di^1_i}{\longrightarrow} ( f_{01},f_{02}, f_{03},f_{12},f_{13},f_{23}) \stackrel{\di^2_i}{\longrightarrow} ( f_{012},f_{013},f_{023},f_{123}) \stackrel{\di^3_i}{\longrightarrow} f_{0123}.
\end{equation}

Each $\Psi\in \mathfrak{X}^\bullet_\text{formal}$ will then be a cocycle if and only if each of its homogeneous components is itself a cocycle. Respectively, $\Psi$ will be a coboundary if and only if each of its homogeneous components is itself a coboundary. For this reason, we now fix all coefficient functions of $\bullet-$ vector fields to be in $V_i$.

\begin{definition}\label{x1x3degdef}
Let $\deg_{x_1x_3}$ denote the sum of degrees in the $(x_1,x_3)$- coordinates of an element in $V$, that is, $\deg_{x_1x_3}(x_0^{k_0}x_1^{k_1}x_2^{k_2}x_3^{k_3})=k_1+k_3$. We write $\deg_{x_1x_3}(f)=c$ if $f\in V$ is a polynomial of $\deg_{x_1x_3}-$ homogeneous terms of degree $c$.
\end{definition}

\begin{remark}\label{x1x3props}
If $\deg_{x_1x_3}(f)=c\in\mathbb{N}_0$, the action of Hamiltonian vector fields (\ref{eq:pisharp-dx0-nearpos}) - (\ref{eq:pisharp-dx3-nearpos}) is related to this degree as follows:

\begin{equation}\label{x1x3}
X_0(f)=cf, \;\;\;
\deg_{x_1x_3}\big(X_0(f)\big)=c, \;\;\;
\deg_{x_1x_3}\big(X_1(f)\big)=c+1,
\end{equation}
\[ \deg_{x_1x_3}\big(X_2(f)\big)=c,\;\;\;\deg_{x_1x_3}\big(X_3(f)\big)=c+1.\]
\end{remark}

\subsubsection{Computation of Poisson cohomology groups with formal coefficients}

Let  $H^m_i(U_Z,\pi)$ denote the $m$--th Poisson cohomology group with coefficients from $V_i$.

\begin{lemma}\label{H0}
 The cohomology group $H^0_i(U_Z,\pi)$ vanishes for all $i> 0$ and 
 \[H^0_0(U_Z,\pi)\simeq\mathbb{R}.\]
 \end{lemma}
  
\begin{proof} Since ${\rm Rank} (\pi)=4$, there are no non-constant Casimirs, so $\ker(\di^0_i)=0$ for every $i>0$. \end{proof}
 
 \begin{lemma}\label{H1}
 The cohomology group $H^1_i(U_Z,\pi)$ vanishes for all $i>0$ and
 \[H^1_0(U_Z,\pi)\simeq\langle\partial_0,\partial_2\rangle\simeq\mathbb{R}^2.\]
 \end{lemma}

 \begin{proof} As a result of Lemma \ref{H0}, it is $\mathrm{Im}(\di^0_i)\cong V_i,\;\dim(\mathrm{Im}(\di^0_i))=r_i, \forall i>0$. The image of $\di^0_i$ is spanned by vector fields of the following forms:

\begin{align}
 \di^0(x_0^{k_0})x_1^{k_1}x_2^{k_2}x_3^{k_3}=
k_0x_0^{k_0-1}x_1^{k_1}x_2^{k_2}x_3^{k_3}X_0&\;\;\;\;\;\;\;\; \;\;\;\;\;\;\;\;\;\;\;\;\;\;\;\;\;\;(\text{type A})
\nonumber
\\
\di^0(x_1^{k_1})x_0^{k_0}x_2^{k_2}x_3^{k_3}=
k_1x_0^{k_0}x_1^{k_1-1}x_2^{k_2}x_3^{k_3}X_1&\;\;\;\;\;\;\;\;\;\;\;\;\;\;\;\;\;\;\;\;\;\;\;\;\;\;(\text{type B})
\nonumber
\\
\di^0(x_2^{k_2})x_0^{k_0}x_1^{k_1}x_3^{k_3}=
k_2x_0^{k_0}x_1^{k_1}x_2^{k_2-1}x_3^{k_3}X_2&\;\;\;\;\;\;\;\;\;\;\;\;\;\;\;\;\;\;\;\;\;\;\;\;\;\;(\text{type C})
\nonumber
\\
\di^0(x_3^{k_3})x_0^{k_0}x_1^{k_1}x_2^{k_2}=
k_3x_0^{k_0}x_1^{k_1}x_2^{k_2}x_3^{k_3-1}X_3&\;\;\;\;\;\;\;\;\;\;\;\;\;\;\;\;\;\;\;\;\;\;\;\;\;\;(\text{type D})
\nonumber 
\end{align}
with $\sum_{s=0}^3k_s=i$. We will show that any $X\in\ker(\di^1_i)$ is written as a linear combination of vector fields of the types $A,B,C,D$.

Let $Y=f_0\partial_0+f_1\partial_1+f_2\partial_2+f_3\partial_3\in\ker(\di^1_i)$, i.e $f_k\in V_i$. Without loss of generality, assume that $f_0\in V_i$ is a monomial with scalar coefficient equal to $1$, so let $f_0=x_0^{k_0}x_1^{k_1}x_2^{k_2}x_3^{k_3}$ and assume $\deg_{x_1x_3}(f_0)=c\neq 0$. 

\noindent Vanishing the coefficient of $\partial_{02}$ in (\ref{d^1exp}) together with (\ref{x1x3}) implies that $f_0$ and $f_2$ must have the same $\deg_{x_1x_3}$ and in particular, 
\[X_2(f_0)=X_0(f_2)=cf_2.\] On the other hand, vanishing the coefficient of $\partial_{01}$ in (\ref{d^1exp}) together with (\ref{x1x3}), one gets that $\deg_{x_1x_3}(f_1)=c+1$ and so
\[X_1(f_0)=\big(\deg_{x_1x_3}(f_1)-1\big)f_1=cf_1.\] Applying the same argument for the coefficient of $\partial_{03}$ we get 
\[X_3(f_0)=cf_3.\]  Given that $k_3=c-k_1, k_2=i-c-k_0$, a direct computation with the formulas $\displaystyle f_j=\frac{1}{c}X_j(f_0),\;j=1,2,3,$ gives respectively

\begin{align}
f_1&=-\frac{k_0}{c}x_0^{k_0-1}x_1^{k_1+1}x_2^{i-c-k_0}x_3^{c-k_1}+\frac{i-c-k_0}{c}x_0^{k_0}x_1^{k_1}x_2^{i-c-k_0-1}x_3^{c-k_1+1},
\label{eq:f1d1}
\\
f_2&=-\frac{k_1}{c}x_0^{k_0}x_1^{k_1-1}x_2^{i-c-k_0}x_3^{c-k_1+1}+\frac{c-k_1}{c}x_0^{k_0}x_1^{k_1+1}x_2^{i-c-k_0}x_3^{c-k_1-1},
\label{eq:f2d1}
\\
f_3&=-\frac{k_0}{c}x_0^{k_0-1}x_1^{k_1}x_2^{i-c-k_0}x_3^{c-k_1+1}-\frac{i-c-k_0}{c}x_0^{k_0}x_1^{k_1+1}x_2^{i-c-k_0-1}x_3^{c-k_1}.
\label{eq:f3d1}
\end{align}

\noindent Splitting the coefficient of $f_0$ as $1=\frac{c-k_1}{c}+\frac{k_1}{c}$, the vector field $Y=f_0\partial_0+f_1\partial_1+f_2\partial_2+f_3\partial_3$ is now  written as

\begin{align}\label{h1np}
Y=& -\left(\frac{k_0}{c}x_0^{k_0-1}x_1^{k_1}x_2^{i-c-k_0}x_3^{c-k_1} \right) X_0
\nonumber
\\
 &- \left( \frac{k_1}{c}x_0^{k_0}x_1^{k_1-1}x_2^{i-c-k_0}x_3^{c-k_1} \right) X_1
\nonumber
\\
 &-\left( \frac{i-c-k_0}{c}x_0^{k_0}x_1^{k_1}x_2^{i-c-k_0-1}x_3^{c-k_1} \right) X_2
\nonumber
\\
 &- \left( \frac{c-k_1}{c}x_0^{k_0}x_1^{k_1}x_2^{i-c-k_0}x_3^{c-k_1-1}  \right) X_3
\end{align}
which is a linear combination of types $A,B,C,D$ and so $Y$ is in $\mathrm{Im}(\di^0_i)$. Thus if $c\neq 0$, it is $\ker(\di_i^1)\subset \mathrm{Im}(\di_i^0)$ for all $i>0$.

The case $\deg_{x_1x_3}(f_0)=0$, is essentially the constant coefficient case.  Indeed, if $f_0$ does not depend on $x_1,x_3$, then $X_0(f_0)=X_2(f_0)=0$ and $\deg_{x_1x_3}(X_1(f_0))=1$. Vanishing the coefficient of $\partial_{01}$, one gets that $X_0(f_1)-X_1(f_0)-f_1=0$. This equation must hold for all terms of the same $\deg_{x_1x_3}-$ degree, so

\[1=\deg_{x_1x_3}(X_1(f_0))=\deg_{x_1x_3}(X_0(f_1)-f_1)\stackrel{(\ref{x1x3})}{\Rightarrow} \deg_{x_1x_3}(f_1)=1.\]
 But then it is $X_0(f_1)-f_1=0$, and so $X_1(f_0)=0$ i.e. $f_0$ does not depend on $x_0, x_2$ either and thus it is constant. The same argument applied on the coefficient of $\partial_{03}$ shows that  $\deg_{x_1x_3}(f_3)=1$.  

\noindent Setting the coefficient of $\partial_{12}$ equal to zero, one gets 
\begin{equation}\label{12}
f_3=X_1(f_2)-X_2(f_1).
\end{equation}
Doing the same for $\partial_{23}$,
\begin{equation}\label{23}
f_1=X_2(f_3)-X_3(f_2).
\end{equation}
Replace (\ref{23}) into (\ref{12}) to get
\begin{equation}\label{f3no}
f_3+X_2(X_2(f_3))=X_2(X_3(f_2)+X_1(f_2).
\end{equation}
Observe that since $\deg_{x_1x_3}(f_3)=1$, a direct calculation shows that $X_2(X_2(f_3))=-f_3$ and so 
\begin{equation}\label{f2ct}
X_2(X_3(f_2)=X_1(f_2).
\end{equation} Now
setting the coefficient of the bivector $\partial_{02}$ in (\ref{d^1exp}) to be zero, we also get that $X_0(f_2)=0$ and so $f_2$ does not depend on $x_1,x_3$.  A direct computation with (\ref{f2ct}) then shows that $f_2$ can only be constant.

Since $\deg_{x_1x_3}(f_1)=\deg_{x_1x_3}(f_3)=1$ assume without loss of generality that $f_1=x_1+ax_3, f_3=bx_1+dx_3$ for $a,b,d\in\mathbb{R}$. Equations (\ref{12}), (\ref{23}) then give $d=1,a=-b$, and as a result, the vector field $f_1\partial_1+f_3\partial_3=X_0-aX_2$ is in $\mathrm{Im}(\di^0)$.   This means that in terms of Poisson cohomology classes, if $Y\in\ker(\di^1)$ with $Y\in\mathfrak{X}^1_0$, then $[Y]$ is determined by $(f_0,f_2)\in\mathbb{R}^2$.   \end{proof}

\begin{lemma}\label{H4}
 The cohomology group $H^4_i(U_Z,\pi)$ vanishes for all $i\geq 0$.
 \end{lemma}

\begin{proof} Obviously, $\ker (\di_i^4)\simeq V_n$ for all $n\geq 0$. Let $n>0$ and 
$$Z=\sum_{0\leq i<j<k\leq 3}f_{ijk}\partial_{ijk}\\ \in\mathfrak{X}^3.$$ 
Suppose $f_{123}$ is a monomial in $V_n$. As in \eqref{x1x3},  $X_0(f_{123})=\deg_{x_1x_3}(f_{123})f_{123}$.  Setting $f_{012}=f_{013}=f_{023}=0$ at (\ref{d^3np}), we get that if $\deg_{x_1x_3}(f_{123})\neq 2$, then 
\[\di^3_i(Z)=\di^3_i(f_{123}\partial_{123})=\big(\deg_{x_1x_3}(f_{123})- 2\big)f_{123}\partial_{0123}.\]
 This shows that 
 $$\{f\partial_{0123}|\;f\in\ V_n, \deg_{x_1x_3}(f)\neq 2\}\subset \rm Im(\di^3_n).$$ 
We want now to show that also $\{f\partial_{0123}|\;f\in\ V_n, \deg_{x_1x_3}(f)= 2\}\subset \rm Im(\di^3_n)$. A direct check shows that given such an $f\in V_n$, one can find polynomials $f_{012}, f_{013}, f_{023}\in V_n$ satisfying 
 $$\deg_{x_1x_3}(f_{012})=\deg_{x_1x_3}(f_{023})=1, \quad  \deg_{x_1x_3}(f_{013})=2,$$
 and such that 
 $$-X_1(f_{023})+X_2(f_{013})-X_3(f_{012})=f.$$ 
 For these $f_{012}, f_{013}, f_{023}, f\in V_n$, the 3-vector field $Z=f_{012}\partial_{012}+ f_{013}\partial_{013}+ f_{023}\partial_{023}+f\partial_{123}$, satisfies $\di^3(Z)=f\partial_{01234}$. As a result,  $\bigoplus_{i>0}H^4_{i}(U_Z,\pi)=0$.
 
 For the constant coefficient case, it is enough to see from (\ref{d^3np}) that  $c\cdot\partial_{0123}=\di^3_0(c\cdot\partial_{123})$.
\end{proof}
 
 \begin{lemma}\label{H2}
 The cohomology group $H^2_i(U_Z,\pi)$ vanishes for all $i> 0$ and
 \[H^2_0(U_Z,\pi)\simeq\langle\partial_0\wedge\partial_2\rangle\simeq\mathbb{R}.\]
 \end{lemma}

\begin{proof} By the previous computation for $H^1_{>0}(U_Z,\pi)$ we have shown that $\dim(\ker(\di^1_i))\\=r_i$ and so $\dim(\mathrm{Im}(\di^1_i))=3r_i$. To prove that $H^2_i(U_Z,\pi)=0$, it is enough to prove that $\dim(\ker (\di^2_i))=3r_i$. We do this by first examining the degree $\deg_{x_1x_3}$ of the equations defining $\ker(\di^2_i)$, that is, vanishing the coefficient functions of the  3-vectors $\partial_{ijk}$ in  (\ref{d^2np}).

\noindent Let $W=\sum_{i=0<j=1}^3f_{ij}\partial_{ij}\in\ker (\di^2_i)$, and suppose that $\deg_{x_1x_3}(f_{01})=c\neq 1$. Set 
$$\deg_{x_1x_3}(f_{13})=\alpha.$$ 
Recall that $X_0(f_{13})=\alpha f_{13}$ and that for a $\deg_{x_1x_3}-$ homogeneous polynomial $f$, it is \[\deg_{x_1x_3}(X_1(f))=\deg_{x_1x_3}(X_3(f))=\deg_{x_1x_3}(f)+1\] 
by (\ref{x1x3}). Given this, and the fact that the coefficient of $\partial_{013}$ in (\ref{d^2np}) 
must satisfy the equation $(\alpha-2)f_{13}- X_1(f_{03})+X_3(f_{01})=0$, one then gets that necessarily 
\[\deg_{x_1x_3}(f_{03})=c\;\;\;\textrm{and}\;\;\;c+1=\alpha.\]

Then we turn to the coefficient of $\partial_{012}$. By the previous argument, $\deg_{x_1x_3}(X_2(f_{01}))\\=\deg_{x_1x_3}(f_{03})=c$. Suppose then that the other two terms are of different $\deg_{x_1x_3}$, so let
\begin{align}\label{c''}
\deg_{x_1x_3}(f_{12})=&\beta,
\\
\nonumber
\deg_{x_1x_3}(X_1(f_{02}))=\deg_{x_1x_3}(f_{02})+1=&\beta.
\end{align}

\noindent Vanishing the coefficient of $\partial_{023}$ one first gets the known fact 
\[\deg_{x_1x_3}(f_{01})=\deg_{x_1x_3}(X_2(f_{03}))=c.\]
 Furthermore, the degrees of the other terms of the coefficient function of $\partial_{023}$ must satisfy the equation
\[\deg_{x_1x_3}(f_{23})=\deg_{x_1x_3}(X_3(f_{02}))=\deg_{x_1x_3}(f_{02})+1.\]
By the assumption (\ref{c''}), this integer is equal to $\beta$. We thus have the following sets of equations with respect to the degree $\deg_{x_1x_3}$ of the coefficient functions $f_{ij}$ of a bivector $W\in \ker (\di^2_i):$

 \begin{equation}\label{x1x3d2first}
 \deg_{x_1x_3}(f_{23})=\deg_{x_1x_3}(f_{12})=\deg_{x_1x_3}(f_{02})+1 =\beta, 
 \end{equation}
 \[\deg_{x_1x_3}(f_{13})=\deg_{x_1x_3}(f_{01})+1=\deg_{x_1x_3}(f_{03})+1 =c+1.\]

Vanishing the coefficient of $\partial_{123}$ we then get that $c=\beta$. Equations (\ref{x1x3d2first}) then become
 \begin{equation}\label{x1x3d2}
 \deg_{x_1x_3}(f_{01})= \deg_{x_1x_3}(f_{03})=\deg_{x_1x_3}(f_{23})= \deg_{x_1x_3}(f_{12})=c, 
 \end{equation}
 \[\deg_{x_1x_3}(f_{02})=c-1,\;\deg_{x_1x_3}(f_{13})=c+1.\]
 
 Now set again the coefficients of $\partial_{012},\partial_{013},\partial_{023}$ in (\ref{d^2np}) to be equal to $0$. Solving each equation  respectively for $f_{12}, f_{13}, f_{23}$ and with the help of (\ref{x1x3d2}) and (\ref{x1x3}), we get
 
\begin{align}
f_{12}&=\frac{1}{c-1}\big[X_1(f_{02})-X_2(f_{01})-f_{03}\big],
\label{eq:f12d2}
\\
f_{13}&=\frac{1}{c-1}\big[X_1(f_{03})-X_3(f_{01})\big],
\label{eq:f13d2}
\\
f_{23}&=\frac{1}{c-1}\big[X_2(f_{03})-X_3(f_{02})-f_{01}\big].
\label{eq:f23d2}
\end{align}
Replacing $f_{12}, f_{13}, f_{23}$ in the coefficient of $\partial_{123}$ in (\ref{d^2np}), one has

\begin{align}
\nonumber
&X_1(f_{23})-X_2(f_{13})+X_3(f_{12})=\\
\nonumber
&[X_2,X_3](f_{01})-X_1(f_{01})-[X_1,X_3](f_{02})+[X_1,X_2](f_{03})-X_3(f_{03})=0,
\end{align}

\noindent where for the second equality we used that $\{x_2,x_3\}=x_1, \{x_1,x_3\}=0, \{x_1,x_2\}=x_3$ by (\ref{eq:Near-Positive}). Thus $\dim(\ker(\di^2_i))=3r_i$ and since $\dim(\mathrm{Im}(\di^1_i)=3r_i$, we get that $H^2_i(U_Z,\pi)=0$ for $i>0$.

To cover the remaining case, suppose $\deg_{x_1x_3}(f_{01})=1$. This corresponds to the constant coefficient case. Indeed, by (\ref{x1x3d2}),
\begin{equation}\label{deg2}
\deg_{x_1x_3}(f_{01})=\deg_{x_1x_3}(f_{03})=\deg_{x_1x_3}(f_{23})=\deg_{x_1x_3}(f_{12})=1,
\end{equation}
 \[\deg_{x_1x_3}(f_{02})=0,\;\;\textrm{and}\;\;  \deg_{x_1x_3}(f_{13})=2.\]

\noindent Since $W\in\ker (\di_i^2)$, the equations  satisfied by the coefficient functions $f_{ij}$ are then
\begin{align}\label{c1}
&X_1(f_{02})-X_2(f_{01})-f_{03}=0,\\
\nonumber
&X_1(f_{03})-X_3(f_{01})=0,\\
\nonumber
&X_2(f_{03})-X_3(f_{02})-f_{01}=0,\\
\nonumber
&X_1(f_{23})-X_2(f_{13})+X_3(f_{12})=0.
\end{align}

The first three equations of (\ref{c1}) together with (\ref{deg2}) imply that the vector field $f_{01}\partial_1+f_{02}\partial_2+f_{03}\partial_3$ is in $\ker (\di_i^1)$. By triviality of the first cohomology group proved in Lemma \ref{H0} (see in particular the case $\deg_{x_2x_3}(f_0)=0$ there), the function $f_{02}$, being the coefficient of $\partial_2$, is constant. Furthermore, by the similar discussion at the end of the proof of  Lemma \ref{H0}, it is $f_{01}=x_1+ax_3, f_{03}=-ax_1+x_3,\; a\in\mathbb{R}$. A direct check shows that 
\[\di^1\big( (x_0-ax_2)\partial_0\big)=f_{01}\partial_{01}+f_{03}\partial_{03},\]  i.e. $f_{01}\partial_{01}+f_{03}\partial_{03}\in \rm Im(\di^1_i)$.

 For the other part of $W$, i.e. $f_{12}\partial_{12}+f_{13}\partial_{13}+f_{23}\partial_{23}$, recall that $\deg_{x_1x_3}(f_{23})=\deg_{x_1x_3}(f_{12})=1$ and $\deg_{x_1x_3}(f_{13})=2$. Then, the coefficient $f_{123}=X_1(f_{23})-X_2(f_{13})+X_3(f_{12})$ of $\partial_{123}$ in (\ref{d^2np})  has $\deg_{x_1x_3}(f_{123})=2$. Thus the 3-vector field $f_{12}\partial_{012}+f_{13}\partial_{013}+f_{23}\partial_{023}$ is in $\rm \ker (\di^3_i)$ because of the last equation in (\ref{c1}). By Lemma \ref{H4}, it is $\dim(\mathrm{Im}(\di^3_i))=r_i$ and so $\dim(\ker(\di^3_i))=3r_i$. The space of solutions for the last equation of (\ref{c1}) is then also of dimension $3i$.

\noindent As a result, we have showed that all elements of  $\ker (\di_i^2)$ belong in  $\rm Im (\di_i^1)$, except for constant multiples of the generator $\partial_0\wedge\partial_2$. \end{proof}

 \begin{lemma}\label{H3}
 The cohomology group $H^3_i(U_Z,\pi)$ vanishes for all $i\geq 0$.
 \end{lemma}
\begin{proof} 
Let $i>0$. By the previous computation for $H^2_i(U_Z,\pi)$ in Lemma \ref{H2}, it is $\dim(\ker (\di^2_i))=3r_i$, and so $\dim(\mathrm{Im}(\di^2_i))=3r_i$. By Lemma \ref{H4}, it is $\dim(\mathrm{Im}(\di^3_i))=\dim(\ker (\di^4_i))=r_i$ and so $\dim(\ker (\di^3_i))=3r_i$. Thus $H^3_i(U_Z,\pi)=0,\;\;\forall i>0$. 

For the constant coefficient case, by (\ref{d^3np}) it is obvious that a constant coefficient 3-vector field in $\ker (\di^3_0)$ is of the form  $c_{012}\partial_{012}+c_{013}\partial_{013}+c_{023}\partial_{023}$ and this is equal to $-\di^2_0\big(c_{012}\partial_{12}+\frac{c_{013}}{2}\partial_{13}+c_{023}\partial_{23}\big)$. \end{proof}

As a conclusion,  when the coefficients belong to some $V_i$ with $i>0$ fixed,  (\ref{seq}) becomes an exact sequence.  

Let $H^{\bullet}_{\textnormal{formal}}(U_{Z},  \pi)$ be the cohomology of the cochain complex $(\mathfrak{X}^\bullet_\text{formal},\di^\bullet)$. We furthermore have the following.

\begin{proposition}\label{prop:near-pos-cohom}
Let $(M, \omega)$ be a near-symplectic 4--manifold. Consider  the tubular neighbourhood  $(U_Z, \pi)$ of the singular locus $Z_{\omega}$ equipped with the Poisson bivector \eqref{eq:Near-Positive}.  Assume $Z_{\omega}$ has only one component. Then

\begin{tabular}{lclcl}
$H^0_{\textnormal{formal}}(U_{Z},  \pi) \cong\mathbb{R} \cong \langle 1  \rangle$ 		
\nonumber
\\
$H^1_{\textnormal{formal}}(U_{Z},  \pi) \cong \mathbb{R}^{2}\cong\langle \partial_0\, ,  \,\partial_{2}  \rangle $
\nonumber
\\
$H^2_{\textnormal{formal}}(U_{Z},  \pi) \cong \mathbb{R} \cong \langle \partial_{0} \wedge \partial_{2} \rangle $			
\nonumber
\\
$H^3_{\textnormal{formal}}(U_{Z},  \pi)=  0 $
\nonumber
\\
$H^4_{\textnormal{formal}}(U_{Z},  \pi)= 0 $
\nonumber
\end{tabular}

\end{proposition}

\begin{proof} Since the operators $\di^\bullet$ are linear, it suffices to compute the cohomology spaces $H^m_i(U_Z,\pi)$ which was done in Lemmata \ref{H0} - \ref{H3} .  Due to (\ref{split}), one may then replace $V_i$ by $V_\textrm{formal}$, the algebra of formal power series equipped with (\ref{eq:Near-Positive}), since $V_{\rm formal}=\bigoplus_{i\geq0}V_i$.
\end{proof}

\subsubsection{From formal to smooth coefficients}
To use the previous Proposition for the computation of Poisson cohomology with smooth coefficients, we need the following observation.
 
\begin{lemma}\label{smoothx0x2}
The Poisson cohomology computation in Proposition \ref{prop:near-pos-cohom}  extends to the Poisson cohomology with coefficient functions that are smooth in $x_0,x_2$ and formal in $x_1,x_3$.
\end{lemma}

\begin{proof}
For completeness we present here a proof of this fact for $H^1(U_Z,\pi)$. Consider vector fields $Y=\sum_{i=0}^3f_i\partial_i\in \ker ({\di^1})$  whose coefficient functions $f_i$ are written as $f_i=\sum_kg_{ik}(x_0,x_2)x_1^{a_{ik}}x_3^{b_{ik}}$, with $g_{ik}$ smooth. The notion of $\deg_{x_1x_3}-$ degree is also valid for such coefficient functions and one can use Definition \ref{x1x3degdef} and properties in Remark \ref{x1x3props} in the same manner as for formal or polynomial functions.

\noindent To keep the notation and reasoning of the proof of Lemma \ref{H1}, let $f_0=gx_1^{k_1}x_3^{k_3}$ with $g=g(x_0,x_2)$ smooth and $\deg_{x_1x_3}(f_0)=k_1+k_3=c\neq 0$. Then using the same arguments, equations (\ref{eq:f1d1})-(\ref{eq:f3d1}) become
\begin{align}
f_1&=-\frac{\partial_0(g)}{c}x_1^{k_1+1}x_3^{k_3}+\frac{\partial_2(g)}{c}x_1^{k_1}x_3^{k_3+1},
\label{eq:f1d1s}
\\
f_2&=-\frac{g}{c}\big[k_1x_1^{k_1-1}x_3^{k_3+1}-k_3x_1^{k_1+1}x_3^{k_3-1}\big],
\label{eq:f2d1s}
\\
f_3&=-\frac{\partial_0(g)}{c}x_1^{k_1}x_3^{k_3+1}-\frac{\partial_2(g)}{c}x_1^{k_1+1}x_3^{k_3}.
\label{eq:f3d1s}
\end{align}
Writing $f_0$ as $f_0=\frac{k_1+k_3}{c}gx_1^{k_1}x_3^{k_3}$, the vector field $Y\in\ker (\di^1)$ is written as
\begin{align}\label{smooth-formal}
Y=&-\frac{\partial_0(g)}{c}x_1^{k_1}x_3^{k_3}X_0 -\frac{k_1g}{c}x_1^{k_1-1}x_3^{k_3}X_1
\\
 & -\frac{\partial_2(g)}{c}x_1^{k_1}x_3^{k_3}X_2
-\frac{k_3g}{c}x_1^{k_1}x_3^{k_3-1}X_3.
\nonumber
\end{align}
Recalling the identity 
\[\di^0(f)=\sum_{k=0}^3X_k(f)\partial_k=-\sum_{k=0}^3\partial_k(f)X_k\] from \eqref{d^0np}, equation  \eqref{smooth-formal} then reads
\[Y= \di^0(\frac{1}{c}f_0).\] Thus for Poisson cohomology classes with such coefficients, $[Y]= 0$. 

For the case $\deg_{x_1x_3}(f_0)=0$, let $Y=\sum_{i=0}^3f_i\partial_i\in \ker ({\di^1})$ with $f_0(x_0,x_2)$ being a smooth function. The corresponding arguments of Lemma \ref{H0} show in exactly the same way that $f_0$ and $f_2$ are constant. Since $\deg_{x_1x_3}(f_1)=\deg_{x_1x_3}(f_3)=1$, assume that $f_1=g_{11}x_1+g_{13}x_3,\;f_3=g_{31}x_1+g_{33}x_3$ with $g_{ik}$ being smooth functions of $x_0,x_2$. Equations \eqref{12},\eqref{23} imply that $g_{11}=g_{33}$, $g_{31}=-g_{13}$ and so the vector field $Y^{'}=f_1\partial_1+f_3\partial_3$ satisfies
\[Y^{'}=\di^0(-g_{11}x_0+g_{13}x_2),\]
and is so the image of a function of zero $\deg_{x_1x_3}-$ degree that is smooth on $x_0,x_2$.

The corresponding proofs for the fourth, second and third cohomology groups with coefficients smooth in $x_0,x_2$ and polynomial in $x_1, x_3$, follow the proofs of Lemmata \ref{H4}, \ref{H2} and \ref{H3} respectively using similar arguments. \end{proof}

\begin{definition}\label{flat x0x2}
Define a function $f\in C^\infty(U_Z)$ to be {\em flat} if all its derivatives and the function itself vanish along the singular locus $\{x_1=x_3=0\}$ of (\ref{eq:Near-Positive}). 
\end{definition}
\begin{remark}\label{smooth x0x2}
Let $\mathfrak{X}^\bullet_\textnormal{flat}(U_Z)$, $\mathfrak{X}^\bullet_\textnormal{formal}(U_Z)$ and $\mathfrak{X}^\bullet_\textnormal{smooth}(U_Z)$ be the multivector fields with flat, formal and smooth coefficients respectively.  By a consequence of Borel's theorem, the sequence

\[0\longrightarrow \big(\mathfrak{X}^k_{\textnormal{flat}}(U_Z),\di^k\big)\longrightarrow \big(\mathfrak{X}^k_{\textnormal{smooth}}(U_Z),\di^k\big)\longrightarrow\big(\mathfrak{X}^k_{\textnormal{formal}}(U_Z),\di^k\big)\longrightarrow 0\]
is exact for any $k=0, 1, 2, 3$.
\end{remark}

 \noindent We now compute the smooth Poisson cohomology using an idea of Ginzburg \cite{G96}.

\begin{proposition}\label{smooth np}
The smooth Poisson cohomology of the Poisson bivector (\ref{eq:Near-Positive}) on $U_Z$ is given in Proposition \ref{prop:near-pos-cohom}.
\end{proposition}

\begin{proof}
Due to Lemma \ref{smoothx0x2} and Remark \ref{smooth x0x2},  it suffices to show that the flat cohomology $H^\bullet_\textnormal{flat}(U_Z,\pi)$ vanishes. Extend $\pi^\sharp$ to the chain map 
$$\wedge^\bullet\pi^\sharp\colon  \big( \Omega^\bullet(U_Z), \di_{\mathrm{dR}}\big)\longrightarrow \big(\mathfrak{X}^\bullet(U_Z), \di_\pi\big)$$
and then consider the restriction to forms with flat coefficients 
$$\wedge^\bullet\pi^\sharp_\textnormal{flat}\colon \Omega^\bullet_\textnormal{flat}(U_Z)\longrightarrow \mathfrak{X}^\bullet_\textnormal{flat}(U_Z).$$

\noindent Away from the singular locus, $\pi^\sharp_\textnormal{flat}$ is an isomorphism. Indeed, 
\[
\pi^\sharp_\textnormal{flat}\left(\sum_{i=0}^3f_i\di x_i\right)=0\Leftrightarrow\begin{cases}  x_1f_1+x_3f_3=0,\\
-x_1f_0+x_3f_2=0,\\
-x_3f_1+x_1f_3=0,\\
-x_3f_0-x_1f_2=0.
  \end{cases} 
\]
Solving the first and third equation above, one gets that outside the singular locus, $f_1=f_3=0$. Similarly the second and fourth equations imply that $f_0=f_2=0$ and so $\pi^\sharp_\textnormal{flat}$ is injective. On the other hand, if $\sum_{i=0}^3g_i\partial_i$ is in the image of $\pi^\sharp_\textnormal{flat}$ then there is always a flat preimage $\sum_{i=0}^3f_i\di x_i$ with
\[f_0=\frac{x_1g_1+x_3g_3}{x_1^2+x_3^2},\;f_1=\frac{-x_1g_0+x_3g_2}{x_1^2+x_3^2},\;f_2=\frac{-x_3g_1+x_1g_3}{x_1^2+x_3^2},\;f_3=\frac{-x_3g_0-x_1g_2}{x_1^2+x_3^2}.\]
Finally, the cohomology class of $Y\in \mathfrak{X}^\bullet_\textnormal{smooth}(U_Z,\pi)$ written as a  convergent Taylor series in a neighbourhood of the singular locus is 0, if and only if each $i-$homogeneous term of the Taylor series is itself a coboundary. 
\end{proof}

\begin{remark}\label{dq}
In terms of deformation quantization, the linearity of (\ref{eq:Near-Positive}) implies that one has control on the polynomial degree of each term in the $\star-$ product corresponding to $\pi$.  Let $f,g$ be polynomials of weight $\overline{\omega}(f)$, $\overline{\omega}(g)$ respectively, and $\overline{\omega}(\pi)$ be the weight of the given Poisson structure. As shown in \cite{BP} for the more general case of weight homogeneous Poisson structures, the $k$-th term $B_k(f,g)$ in the Taylor series defining the $\star-$ product, will be of weight $\overline{\omega}(f)+\overline{\omega}(g)-k\overline{\omega}(\pi)$. Given a linear Poisson structure as  (\ref{eq:Near-Positive}) it's easy to see that for the weight vector $\overline{\omega}=(1,1,1,1)$, it is $\overline{\omega}(\pi)=-1$. However a global existence theorem for $\star-$ products over singular spaces is more complicated because of the singularities. With respect to near-symplectic manifolds, a reasonable approach would be through Fedosov's deformation quantization and the use of Whitney functions \cite{PPT12} which are implicitly used in the proof of Proposition \ref{smooth np}. 
\end{remark}

\section{Contact Structures}\label{sec:contact}
In this section we comment on the interaction between the Poisson bivector $\pi_U$ of Proposition \ref{prop:nearsymp} and a contact structure on the tubular neighbourhood of the singular locus $Z_{\omega}$ of a near-symplectic 4-manifold $(M,\omega)$.

Recall that a contact structure on a $(2n-1)$-dimensional manifold $N$ is a maximally non-integrable hyperplane distribution $\xi \subset TN$ determined by the kernel of a globally defined 1-form $\alpha$ satisfying $\alpha\wedge d\alpha^{n-1}\not= 0$. Contact structures on 3-manifolds $N^3$ are classified as {\em tight} or {\em overtwisted}.  A contact structure is called overtwisted if $(N^3, \xi)$ contains an embedding of a disk $D^2 \hookrightarrow N^3$ such that for its characteristic foliation $\Delta = T_p D^2 \cap \xi_p$: $(i)$ the boundary $\partial D^2$ is a closed leaf, and $(ii)$ there is a unique elliptic singular point in the interior of the disk $D^2$. If there is no such a disk, then the contact structure is said to be tight.  It is known that on a near-symplectic 4-manifold $(M,\omega)$ there is an overtwisted contact structure on the boundary of the tubular neighborhood $U_Z$ of the singular locus $Z_\omega$.

\begin{theorem}\cite{GK, Ho04}\label{thm:Honda-contact}
Let $(M, \omega)$ be a near-symplectic 4-manifold.  There is an overtwisted contact structure $\xi = \ker(\alpha)$ on the boundary of the tubular neighbourhood of the degeneracy locus of $\omega$,  $\partial U_Z \cong S^1 \times S^2$ such that $d\alpha = i^{*}\omega$, where $i\colon S^1\times S^2 \hookrightarrow S^1\times D^3$. 
\end{theorem}

Consider again the local model of $\omega$ on $U_Z = S^1\times D^3$ as in \eqref{omega:Darboux}.  Since $\dim(M)=4$ one has that $\omega_Z = 0$. Honda \cite[Sec. 5]{Ho04} provides the following contact form $\alpha$ defining $\xi = \ker(\alpha)$;
\begin{equation}
\alpha = \frac{1}{2}  \left( x_{1}^{2} -2 x_{2}^{2} + x_{3}^{2} \right)dx_0  + x_2 (x_1 dx_3 - x_3 dx_1).
\end{equation}
Now we look at the action of $\pi^{\sharp}$ on this contact form. Consider the Poisson bivector $\pi_U = \eta + \ast \eta$ on $U_Z$ as in \eqref{pi:U} and \eqref{pi:near-pos}, and the Hamiltonian vector fields \eqref{eq:pisharp-dx0-nearpos}-\eqref{eq:pisharp-dx3-nearpos}. Then
$$
\pi^{\sharp}_U(\alpha) =   \frac{1}{2}  \left( x_{1}^{2} -2 x_{2}^{2} + x_{3}^{2} \right) X_0 - x_2 (x_1^2 + x_3^2) \partial_2,
$$
and after a change of coordinates $x_0 = x_0, x_1 = r \cos(\phi), x_2 = z,  x_3 = r \sin (\phi)$, the previous expression becomes
$$
\pi^{\sharp}_U(\alpha) =     \left( \frac{1}{2} r^2 - z^{2} \right) X_0- \left( r^2 \cdot z\right)  \partial_z.
$$
%
This vector field is clearly zero on $Z_{\omega} = S^1\times \lbrace \mathbf{0} \rbrace$. As it moves to the boundary $S^1 \times S^2$, the action of $\pi^{\sharp}$ on the contact form is a combination of the Hamiltonian vector field $X_0$ and the Poisson vector field $\partial_z$.  

In \cite{Ho04} the author also provides the Reeb vector field of the contact structure, i.e. the unique vector field $Y$ such that $Y\in \ker(d\alpha)$ and $\alpha(Y) =1$. Up to a multiple, the Reeb vector field is given by  
$$
Y = \frac{1}{f} \left(x_1^2 - 2x_2^2 + x_3^2 \right) \partial_0 + 3x_2 \left(-x_3 \partial_1 + x_1 \partial_3\right)
$$
where $f = -\frac{1}{2}\left[ ( x_1^2 + x_3^2 ) (x_1^2 - 2x_2^2 + x_3^2) + 4x_2^4 \right]$. Since the modular vector field $V_{\mathrm{mod}}$ of the Poisson structure $\pi_U$ is $2\partial_0$, the Reeb vector field can be expressed using the modular vector field and a Hamiltonian vector field
\begin{equation}\label{Reeb}
Y = \frac{1}{2f} \left(x_1^2 - 2x_2^2 + x_3^2 \right) V_{\mathrm{mod}} + \left(3x_2\right) X_2.
\end{equation}

\noindent Denote by $\mathrm{pt}_N:= \lbrace(0, 1, 0) \rbrace, \mathrm{pt}_S:= \lbrace(0, -1, 0) \rbrace$ the north and south poles of $S^2$.  The closed orbits of the Reeb vector field are  
\begin{equation}\label{orbits}
\mathcal{O}_1=S^{1}\times \lbrace \mathrm{pt}_N \rbrace \quad, \quad \mathcal{O}_2=S^{1}\times \lbrace \mathrm{pt}_S \rbrace \quad , \quad \mathcal{O}_3=S^{1}\times \lbrace(x_1, 0, x_3) \rbrace
\end{equation}
with $x_1^2 + x_3^2 = 1$ and $x_1, x_3$ fixed.  Hence, along its closed orbits, $Y$ is a constant multiple of the modular vector field $V_{\mathrm{mod}}$ as 
\begin{equation}
Y=\frac{1}{4}V_{\mathrm{mod}}\;\;\text{on}\;\;\mathcal{O}_1 \quad, \quad Y=\frac{1}{4}V_{\mathrm{mod}}\;\;\text{on}\;\;\mathcal{O}_2 \quad, \quad Y=-\frac{1}{2}V_{\mathrm{mod}}\;\text{on}\;\;\mathcal{O}_3
\end{equation}
 respectively. By Proposition \ref{prop:near-pos-cohom} one can summarize the previous observations in the following corollary. 

\begin{corollary}
Let $(M,\omega)$ be a near-symplectic 4--manifold and $\pi$ the Poisson structure on the tubular neighborhood of the zero locus. Along closed orbits, the Reeb vector field of the contact structure $(\partial U_Z, \xi)$ (as in Theorem \ref{thm:Honda-contact}) is in the Poisson cohomology class of the modular vector field $[\partial_0] \in H_{\pi}^1(M)$.
\end{corollary}

In higher dimensions, the situation is unknown. On one hand, it is not clear if there is a contact structure in some submanifold of a near-symplectic manifold. On the other, the Poisson cohomology for $(M, \pi)$ would require other techniques for its computation.

\noindent {\bf Acknowledgements} 
We warmly thank Pedro Frejlich and Ralph Klaasse for their comments and feedback on drafts of this work. We are also very grateful to A{\"i}ssa Wade for fruitful discussions and interest in this work. Our thanks extend also to Viktor Fromm, Luis Garc{\'i}a-Naranjo, Alexei Novikov, Tim Perutz and Pablo Su{\'a}rez-Serrato. 

\noindent{\bf Data Availability}
Data sharing is not applicable to this article as no new data were created or analyzed in this study.


\end{document}